\documentclass[12pt,reqno]{amsart}
\usepackage{amssymb,amsmath,enumerate,extarrows,mathrsfs}
\usepackage{fullpage}
\usepackage[colorlinks=true,linkcolor=black,citecolor=black]{hyperref}
\usepackage[all]{xy}

\newcommand{\mm}{\mathfrak m}

\newcommand{\Fc}{\mathcal{F}}
\newcommand{\Lscr}{\mathscr{L}}

\DeclareMathOperator{\Coker}{Coker}
\DeclareMathOperator{\creg}{creg}

\DeclareMathOperator{\Img}{Im}

\DeclareMathOperator{\Ker}{Ker}
\DeclareMathOperator{\pdim}{pd}
\DeclareMathOperator{\reg}{reg}
\DeclareMathOperator{\Sym}{Sym}
\DeclareMathOperator{\Tor}{Tor}
\DeclareMathOperator{\Ann}{Ann}

\newtheorem{thm}{\bf Theorem}[section]
\newtheorem{lem}[thm]{\bf Lemma}
\newtheorem{cor}[thm]{\bf Corollary}
\newtheorem{prop}[thm]{\bf Proposition}

\theoremstyle{definition}
\newtheorem{quest}[thm]{Question}
\newtheorem{defn}[thm]{\bf Definition}

\theoremstyle{remark}
\newtheorem{rem}[thm]{Remark}
\newtheorem{ex}[thm]{Example}

\theoremstyle{remark}
\newtheorem*{rem*}{Remark}

\numberwithin{equation}{section}
\title[Products of ideals]{Products of ideals of linear forms in  quadric hypersurfaces }
\author[A. Conca]{Aldo Conca}
\address{Dipartimento di Matematica, Universit\`a Degli Studi di Genova, Via Dodecaneso 35, 16146 Genova, Italy}
\email{conca@dima.unige.it}
\author[H.D. Nguyen]{Hop D. Nguyen}
\address{Dipartimento di Matematica, Universit\`a Degli Studi di Genova, Via Dodecaneso 35, 16146 Genova, Italy}
\email{ngdhop@gmail.com}
\author[T. Vu]{Thanh Vu}
\address{Hanoi University of Science and Technology, 1 Dai Co Viet, Hai Ba Trung, Hanoi, Vietnam}
\email{vuqthanh@gmail.com}

\subjclass[2010]{13D02, 13D05}
\keywords{Linear resolution; regularity; Koszul algebra; universally Koszul algebra.}

\begin{document}

\begin{abstract}
Conca and Herzog proved in \cite{CH} that any product of ideals of linear forms in a polynomial ring has a linear resolution. The goal of this paper is to establish the same result for any quadric hypersurface. The main tool we develop and use is a flexible version of Derksen and Sidman's approximation systems  \cite{DS2}.  \end{abstract}

\maketitle

\section{Introduction}
Let $R$ be a standard graded algebra over a field $k$, i.e., $R=S/J$ where $S=k[x_1,\dots,x_n]$ is a graded polynomial ring with variables in degree 1 and $J$ a homogeneous ideal. For a finitely generated graded $R$-module $M$, denote by $\reg_R M$ the Castelnuovo-Mumford regularity of $M$ over $R$. The Castelnuovo--Mumford regularity computed over $S$ is denoted by $\reg M$ instead of $\reg_S M$. By definition $R$ is {\it Koszul} if $\reg_R k=0$. Koszul algebras are known to share with polynomial rings several  important features. For example, a beautiful and elegant result of Avramov, Eisenbud and Peeva \cite{AE,AP} asserts that Koszul algebras are exactly the algebras over which every finitely generated graded module has a finite Castelnuovo-Mumford regularity.  Another very special feature of Koszul algebras is that they tend to have  syzygies of relatively low degrees \cite{ACI}. But not all Koszul algebras are born equal and some of them are better than others. For example, Roos says in \cite{R} that a Koszul algebra R is `good' if all the Poincar\'e series of finitely generated graded $R$-modules are rational functions sharing a common denominator. He provides in the same paper examples of good and bad Koszul algebras. 

Other important notions are  that of absolutely Koszul algebras, see \cite{CINR, HI,IR}, and of  universally Koszul algebras, see \cite{Con1,Con2}. The latter play an important role in our discussion so let us recall the definition. Let $\Lscr(R)$ be  the collection of ideals generated by  linear forms of $R$,  and  $\Lscr^\infty(R)$ be the collection of products of ideals generated by  linear forms of $R$. 
The algebra $R$ is said to be {\it universally Koszul} if every ideal of $\Lscr(R)$ has a linear resolution over $R$. We refer the reader to \cite{Con1, Con2} for basic facts about universally Koszul algebras and for the characterization of universally Koszul Cohen-Macaulay domains and universally Koszul algebras defined by monomials. 
In \cite{CH} the following result was proved.  
\begin{thm}[Conca--Herzog, {\cite[Theorem 3.1]{CH}}]
\label{thm_ConcaHerzog}
Let $S$ be a polynomial ring over $k$. Then any   ideal $I$ of $\Lscr^\infty(S)$ has a linear resolution. 
\end{thm}
This result suggests the following definition.  

\begin{defn} 
We say that $R$ is {\it universally$^*$ Koszul} if every ideal in $\Lscr^\infty(R)$ has a linear resolution over $R$.
\end{defn} 

In  \cite[Section 5, Question 9]{CDR} the authors ask:   
\begin{quest}
\label{quest_product} 
Are universally Koszul algebras necessarily universally$^*$  Koszul? 
\end{quest}

A positive answer is given in  \cite[Section 5, Theorem 11]{CDR} for algebras of small dimension. However, Question \ref{quest_product} remains open in general. The main result of this paper answers in the positive Question \ref{quest_product} for  quadric hypersurfaces (which are known to be  universally Koszul). 

\begin{thm}[Theorem \ref{thm_quadric}]
\label{thm_main_intro}
Any quadric hypersurface is universally$^*$  Koszul.    
\end{thm} 

To prove the theorem we  introduce an invariant,  the {\it Chardin mixed regularity} $\creg_S(M,N)$ of a pair of graded $S$-modules $M$ and $N$, as follows:
$$\creg_S(M,N)=\sup  \{ \reg \Tor_i^S(M,N)-i  : i\geq 0\}.$$ 
We say that $R$  is {\it $\Tor$-linear} with respect to a collection $\Fc$ of ideals of $S$ if 
$$
\creg_S(R,S/I)\leq \reg(S/I)+1 \quad \text{for every $I\in \Fc$}.
$$
Then we observe in Proposition \ref{prop_sufficient_condition} that: 
 \begin{enumerate}
\item  If $R$ is   $\Tor$-linear with respect to $\Lscr(S)$ then $R$ is universally Koszul, 
\item  If $R$ is   $\Tor$-linear with respect to $\Lscr^\infty(S)$ then $R$ is universally$^*$ Koszul. 
 \end{enumerate}
Finally we prove the following:
    
\begin{thm}
\label{thm_main_Torlin}
If $R=S/(f)$ is a quadric hypersurface \textup{(}namely $\deg f =2$\textup{)}, then $R$ is $\Tor$-linear with respect to  $\Lscr^\infty(S)$.
\end{thm} 

We deduce Theorem \ref{thm_main_Torlin} from Theorem \ref{thm_reg_quotient}, which asserts that 
\[
\creg_S(S/I, S/(f))=\reg S/I+\reg S/(f) 
\]
where $I\in \Lscr^\infty(S)$ and $f$ is a homogeneous element of $S$. Our proof of this equality uses a generalization of an inductive method due to Derksen and Sidman \cite{DS2}. Roughly speaking, Derksen and Sidman's idea is to keep the regularity of a module $N$ under control by means of an approximation system, that is a family of surjective maps $\phi_i:N\to N_i$ such that $\sum_i \Ann (\Ker \phi_i)$ is (primary to) the graded maximal ideal of $S$.  For the applications we have in mind, we need a more flexible mechanism where the maps $\phi_i$ need not be  surjective.  This gives rise to the notion of generalized approximation system  that we introduce and discuss in Section \ref{sect_generalizedapprox}. 

This paper is the continuation of an effort started by the last two authors in \cite{NgV2} with the goal of  understanding resolutions of products of ideals generated by linear forms in quotient rings.

\section{Regularities}
\label{sect_pre}
 
Let $k$ be a field and let $R$ be a standard graded $k$-algebra. In other words, $R$ is a commutative, $\mathbb N$-graded algebra with $R_0=k$ and $R$ is generated over $k$ by finitely many elements of degree $1$. As such, $R$ can be represented as $R=S/J$ where $S$ is a polynomial ring (i.e.~the symmetric algebra of $R_1$) and $J$ a homogeneous ideal of $S$.  We also denote by $k$ the $R$-module $R/\oplus_{i>0}R_i$. Let $M$ be a finitely generated graded $R$-module. The {\it Castelnuovo-Mumford regularity} of $M$ over $R$ is
\[
\reg_R M=\sup\{j-i: \Tor^R_i(k,M)_j\neq 0\},
\]
if $M\neq 0$ and, by convention, $\reg_R 0=-\infty$. We say that $M$ has a {\it linear resolution} over $R$ if for some integer $d$, the equality $\Tor^R_i(k,M)_j=0$ holds for all $j\neq i+d$. In that case, we also say that $M$ has a {\it $d$-linear resolution} over $R$.

Let $\mm$ be the graded maximal ideal of $S$. The Castelnuovo-Mumford regularity  of $M$ computed over the polynomial ring $S$ has also a well-known  cohomological  interpretation:
\[
\reg_S M=\sup\{i+j: H^i_{\mm}(M)_j\neq 0\}.
\]
Here $H^i_{\mm}(M)$ denotes the $i$-th local cohomology of $M$ with support in $\mm$. Because of this, the regularity computed over $S$ is sometimes called the absolute  Castelnuovo-Mumford regularity and it is  denoted  simply by $\reg M$   while that over $R$ is called the relative Castelnuovo-Mumford regularity.  Throughout the paper we will use without reference the fact that (absolute/relative) Castelnuovo-Mumford regularity behaves well with respect to short exact sequences (see \cite[Section 2, Lemma 9]{CDR} for details). 
We will also use the following well-known fact (see \cite[Lemma 1.1]{CH}). 

\begin{lem}
\label{lem_finitelength}
Let $M$ be a graded $R$-module of finite length.  Then $\reg M=\sup\{i: M_i\neq 0\}.$
If furthermore $0\to M\to P\to N \to 0$ is  an exact sequence where $P, N$ are finitely generated graded $R$-modules, then $\reg P=\max\{\reg M, \reg N\}.$
\end{lem}

The following result  is important in our  discussion of $\Tor$-linear algebras. 

\begin{prop}
\label{prop_compare_reg}
Let $\varphi:Q\to R$ be a  homomorphism of standard graded Koszul $k$-algebras such that via scalar restriction, $R$ is a finitely generated $Q$-module \textup{(}e.g. $\varphi$ is surjective\textup{)}. Let $M$ and $N$ be finitely generated graded modules over $Q$ and $R$, respectively. Then 
\[
\reg_R (M\otimes_Q N) \le \max \left\{ \reg_Q M+\reg_R N,~ \sup_{i\ge 1}\left\{\reg_Q \Tor^Q_i(M,N)-(i+1)\right\}\right\}.
\]
\end{prop}
\begin{proof}
Let $F$ be the minimal graded free resolution of $M$ over $Q$. Denote $C=F\otimes_Q N$. Then we have
$H_i(C)=\Tor^Q_i(M,N)$ for all $i\ge 0$. Observe that $\reg_R C_i=\reg_Q F_i+\reg_R N$ for all $i$. Now using  \cite[Section 2, Lemma 9(2)]{CDR}, we obtain
\begin{align*}
\reg_R (M\otimes_Q N) &=\reg H_0(C)  \le \max\left\{\sup_{i\ge 0}\{\reg_R C_i-i\},~ \sup_{j\ge 1}\{\reg_R H_j(C)-(j+1)\} \right\} \\
 & = \max\left\{\sup_{i\ge 0}\{\reg_Q F_i-i\}+\reg_R N,~  \sup_{j\ge 1}\{\reg_R H_j(C)-(j+1)\} \right\} \\
 & = \max \left \{\reg_Q M+\reg_R N,~ \sup_{j \ge 1}\{\reg_R \Tor^Q_j(M,N)-(j+1)\}  \right\} \\
 &  \le \max \left \{\reg_Q M+\reg_R N,~ \sup_{j\ge 1}\{\reg_Q \Tor^Q_j(M,N)-(j+1)\}  \right\}.
\end{align*}
The last inequality follows from \cite[Corollary 6.3]{NgV}. 
\end{proof}


\section{Strong versions of Koszulness} 
We keep the following notation for this section: $R$ is a standard graded $k$-algebra with the presentation $R=S/J$ where $S=\Sym(R_1)$ is the symmetric algebra on the $k$-vector space $R_1$. 

By definition,  $R$ is Koszul if $\reg_R k=0$.  For generalities about Koszul algebras and related properties we refer the reader to \cite{CDR}. Here we simply recall the notion of universally Koszul algebras and introduce universally$^*$ Koszul and $\Tor$-linear algebras. 

Let $\Lscr(R)$ be the collection of ideals generated by  linear forms of $R$,  and  $\Lscr^\infty(R)$ be the collection of the products of ideals generated by  linear forms of $R$.  

\begin{defn}   
We say that $R$ is {\it universally Koszul} if every ideal in $\Lscr(R)$ has a linear resolution over $R$. We say  $R$ is {\it universally$^*$  Koszul} if every   ideal in $\Lscr^\infty(R)$  has a linear resolution over $R$. 
\end{defn} 
  
See \cite{Con1,Con2} for generalities on universally Koszul rings. Conca and Herzog's Theorem \ref{thm_ConcaHerzog} asserts that  polynomial rings are universally$^*$ Koszul.
 
We say that $R$ is a {\it quadric hypersurface} if its defining ideal is generated by a quadratic form. By \cite[Proposition 3.1]{Con1}, any quadric hypersurface is universally Koszul.  In \cite[Section 5, Question 9]{CDR} it is asked whether every universally Koszul algebra is indeed universally$^*$  Koszul. To study this question, we now introduce $\Tor$-linear algebras.  For finitely generated graded $R$-modules $M$ and $N$ define the Chardin mixed regularity $\creg_R(M,N)$ of $M$ and $N$  as:   
\[
\creg_R(M,N)=\sup_{i\ge 0}\{\reg_R \Tor^R_i(M,N)-i\}.
\]
Chardin \cite{Ch1} proved  the following
\begin{thm}
\label{thm_Chardin}
Let $M$ and $N$ be finitely generated graded $S$-modules. Then  
\[
\creg_S(M,N) \ge \reg M+\reg N,
\]
and equality holds  if $\dim \Tor^S_1(M,N)\le 1$.
\end{thm}
 
Indeed,  this is a special case of  much more general results in {\it ibid.}. The first assertion follows from  \cite[Lemma 5.1(i)--(ii)]{Ch1} by setting  $S=R$, $F=$ the minimal graded free resolution of $N$ and observing  that condition (B) in {\it loc.~cit.} is trivially satisfied since $F$ is minimal. The second assertion follows from \cite[Corollary 5.8]{Ch1}. Theorem \ref{thm_Chardin} was inspired by previous results on the regularity of  Tor and product of ideals by Sidman \cite{S}, Conca and Herzog \cite{CH}, Caviglia \cite{Ca} and Eisenbud, Huneke and Ulrich \cite{EHU}. 

\begin{defn}
Let $\Fc$ be a collection of ideals of $S$. We say that $R$ is {\it $\Tor$-linear}  with respect to $\Fc$ if 
$\creg_S(S/I,R)\leq \reg(S/I)+1$ for every  $I\in \Fc$. 
\end{defn}
\begin{rem}
Our choice of terminology is justified as follows. Assume that $R$ is $\Tor$-linear with respect to $\Fc$ and $\Fc$ consists of ideals with linear resolutions (which is the case if $\Fc=$ $\Lscr(S)$ or $\Lscr^\infty(S)$). Then for all $I\in \Fc$ and $i\ge 2$, the $S$-module $\Tor^S_i(R,S/I)$ has a linear resolution. Indeed, let $I\in \Fc$ be an ideal with a $d$-linear resolution. By assumption
\[
\reg \Tor^S_i(R,S/I) \le i+\reg(S/I)+1=i+d.
\]
On the other hand, let $F$ be the minimal graded free resolution of $R$ over $S$, then for $i\ge 2$, $F_{i-1}$ is generated in degree at least $i$. Thus $\Tor^S_i(R,S/I)=\Tor^S_{i-1}(R,I)$, being a subquotient of $F_{i-1}\otimes_S I$, is generated in degree at least $i+d$. In particular, it has an $(i+d)$-linear resolution.
\end{rem}

The importance of Tor-linearity is that this property translates questions about resolutions over singular Koszul rings, usually infinite, to questions about the resolutions of Tor over polynomial rings, necessarily finite.
\begin{prop} 
\label{prop_sufficient_condition}
One has:
\begin{enumerate}[\quad \rm (i)]
\item  If $R$ is   $\Tor$-linear with respect to $\Lscr(S)$ then $R$ is universally Koszul and $\reg R\le 1$. 
\item  If moreover $R$ is   $\Tor$-linear with respect to $\Lscr^\infty(S)$ then $R$ is universally$^*$ Koszul. 
 \end{enumerate}
 \end{prop}
 \begin{proof} 
First we prove (ii). Let $I\in \Lscr^\infty(S)$. Using Proposition \ref{prop_compare_reg} for the map $S\to R$, $M=S/I$ and $N=R$, we get
\[
\reg_R (R/IR) \le \max \left\{\reg (S/I),\sup_{i\ge 1} \left\{ \reg \Tor^S_i(S/I,R)-(i+1)\right\}\right\}= \reg S/I.
\]
Hence by Theorem \ref{thm_ConcaHerzog},  $IR$ has a linear resolution over $R$. In other words, $R$ is universally$^*$ Koszul. The proof of the first assertion of (i) is similar. The second assertion of (i) follows from the inequality of Theorem \ref{thm_Chardin} and the hypothesis.
 \end{proof} 
\begin{figure}[ht]
\[
\xymatrix{
\textnormal{Polynomial ring}\ar@{=>}[r]  &\textnormal{$\Tor$-linear w.r.t.~$\Lscr^\infty(S)$} \ar@{=>}[d]^{(1)}  \ar@{=>}[r]^{(2)} &\textnormal{Universally$^*$ Koszul}\ar@{=>}[d]^{(3)} \\
  & \textnormal{$\Tor$-linear w.r.t.~$\Lscr(S)$} \ar@{=>}[dr]  \ar@{=>}[r] &  \textnormal{Universally Koszul} \ar@{=>}[r]                      &   \textnormal{Koszul}  \\ 
& & \reg R\le 1 \ar@{=>}[ru]  &
}
\]
\caption{Properties of a standard graded $k$-algebra $R$} \label{fig_Koszul}
\end{figure}
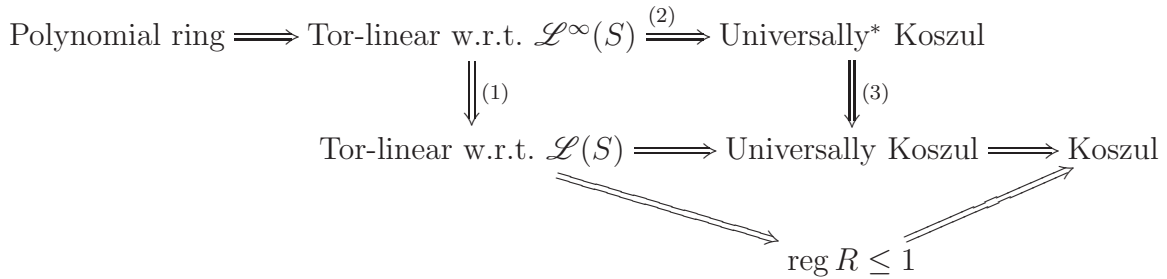
Figure \ref{fig_Koszul} summarizes what we known about the relationships between the properties  we have discussed.  Examples are known showing that all the implications, with the exception of (1), (2), (3), are not reversible.  Note that (3) is related to the open Question \ref{quest_product}.  If the latter turns out to have  a positive answer then (3) is reversible and, in view of Example \ref{ex_nonTor-lin},  $(2)$ is not reversible. We do not know  whether (1) can be reversed.  
\begin{ex}
\label{ex_nonTor-lin} 
Let $A=k[x_1,x_2,x_3,y_1,y_2,y_3]$ and  $R=S/J$ where $J$ is the ideal of $2$-minors of the matrix  
 \[
\left(\begin{matrix}
x_1 & x_2 & y_1 & y_2\\
x_2 & x_3  & y_2 & y_3 
\end{matrix}
\right).
\]
Then $R$ is universally Koszul and $\reg R=1$ by \cite[Theorem 3.3]{Con1}. Nevertheless  $R$ is not $\Tor$-linear with respect to $\Lscr(S)$. Indeed, for $I=(x_1,x_3)\in  \Lscr(S)$ one has 
$$\reg \Tor^S_1(S/(x_1,x_3),R)=3 >2$$
as one can check with  CoCoA \cite{CoCoA},  Macaulay2 \cite{GS}, or by direct computation. 
\end{ex}


\section{Generalized approximation}
\label{sect_generalizedapprox} 
Inspired by \cite{CH} and \cite{DS1}, Derksen and Sidman introduce in \cite{DS2} the approximation systems, a powerful tool to bound inductively the regularity. In the current section, we propose a refinement of Derksen and Sidman's construction, namely generalized approximation. 

In the sequel, $R$ will be a ring and $I\subset R$ an ideal.  We start with an observation concerning the transfer of kernel-cokernel annihilation in tensor products. 

\begin{rem} 
\label{notzero}
Let $\phi:N\to N'$ be an $R$-linear map and $I$ an ideal such that $I\Ker \phi=0$ and $I\Coker \phi=0$. Let $B$ be an $R$-module and $\phi \otimes_R B: N\otimes_R B \to N'\otimes_R B$ the induced map. Then simple homological considerations imply that $I^2\Ker (\phi \otimes_R B)=0$ and $I\Coker (\phi \otimes_R B)=0$ but 
$I\Ker (\phi \otimes_R B)\neq 0$ in general, see \ref{ex_notannihilated}. So one might ask  under which additional assumptions on $\phi$ does it follow that $I\Ker (\phi\otimes_R B)$ vanishes for all $B$. We will see in \ref{kerann} that if $\phi$ is an $I$-approximation (to be defined below) then one can actually deduce that $I\Ker (\phi\otimes_R B)= 0$ for all $B$. 
 \end{rem} 
 
 \begin{ex} 
 \label{ex_notannihilated} 
 Let $R=k[a,b,c]/(a^2,ab)$, $I=(\overline{a})$, $\phi:R^2=Re_1 \oplus Re_2\to R$ the linear map defined by $\phi(e_1)=\overline{a}$ and $\phi(e_2)=\overline{b}$. Then $\Ker \phi=\langle (\overline{a}, 0), (\overline{b}, 0), (0, \overline{a})\rangle$ so that $I\Ker \phi=0$ and 
$I \Coker \phi=0$. Now for $B=R/(ac+bc)$  one has $(\overline{c},\overline{c})\in \Ker (\phi\otimes_R B)$ and $I  (\overline{c},\overline{c}) \neq 0$ in $B^2$.  
 \end{ex} 
 
Next we define the main ingredient to construct generalized approximations.

\begin{defn}
\label{defn_approx}
Let $\phi: N\to N'$ be a homomorphism of $R$-modules. We say that $\phi$ is an {\it $I$-approximation},  if there exist an $R$-module $Z$ and a commutative diagram
\[
\xymatrix{
&   & Z \ar@{->>}[rd]^{\beta} & \\
&N~ \ar@{^{(}->}[ru]^{\alpha} \ar@{-->}[rr]^{\phi} &  &N' 
}
\]
in which $\alpha, \beta$ are $R$-module homomorphisms  such that:
\begin{enumerate}[\quad \rm(i)]
\item $\alpha$ is injective and $\beta$ is surjective,
\item $I\Coker \alpha=0$ and $I\Ker \beta=0$.
\end{enumerate}
If $R$ is a graded ring,  $I$ is a homogeneous ideal and  $\phi$ is a degree $0$  homomorphism of graded $R$-modules, then we require additionally that $Z$ is a graded $R$-module, and  $\alpha$ and $\beta$ are graded of degree $0$.  
\end{defn}

In the proof of our main result in Section \ref{sect_quadric}, we will use the following equivalent formulation of the notion of $I$-approximations.
\begin{prop} 
\label{prop_characterizeI-approx}
Let $\phi: N\to N'$ be a homomorphism of $R$-modules. Then $\phi$ is an $I$-approximation if and only if there exist an $R$-module $Q$ and submodules $M,P,M',P'$ of $Q$ such that the following conditions hold:
\begin{enumerate}[\quad \rm (i)]
\item $M\subseteq P$  and $M'\subseteq P'$,
\item $IM'\subseteq M \subseteq M'$ and  $IP'\subseteq P \subseteq P'$,
\item There exist isomorphisms $N \cong P/M, P'/M'\cong N'$ such that we have a commutative diagram
\[
\xymatrix{
P/M \ar[r]^{\iota}             & P'/M' \ar[d]^{\cong} \\
N \ar[u]^{\cong} \ar@{-->}[r]^{\phi} & N'
}
\]
in which $\iota$ is induced by the inclusions $P\subseteq P'$ and $M\subseteq M'$.
\end{enumerate}
If $R$ is a graded ring, $I$ is a homogeneous ideal, $\phi$ is a homomorphism of graded $R$-modules, then we require additionally that $Q$ is a graded $R$-module, $M,P,M',P'$ are graded submodules of $Q$, and the isomorphisms $N\cong P/M$ and $P'/M' \cong N'$ respect the gradings.
\end{prop}

\begin{proof}
Assume that $\phi$ is an $I$-approximation, so that there exist a module $Z$ and a commutative diagram as in Definition \ref{defn_approx}. Denote $P=\Img \alpha$, $M=0$, $M'=\Ker \beta$, which are submodules of $P'=Z$. Set $Q=Z$. We have a commutative diagram
\[
\xymatrix{
P/0 \ar[r]^{\iota}             & Z/M' \ar[d]^{\overline{\beta}} \\
N \ar[u]^{\alpha} \ar@{-->}[r]^{\phi} & N'
}
\]
where $\overline{\beta}$ is induced by $\beta$. Both vertical maps are isomorphisms. 

Since $\Coker \alpha=Z/P$ and $\Ker \beta=M'$ are annihilated  by $I$, we have
\begin{align*}
IZ &\subseteq P \subseteq Z,\\
IM' &=0 \subseteq M'. 
\end{align*}
Therefore the modules $Q,M,P,M',P'$ satisfy the required conditions.

Conversely, assume that there exist modules $Q,M,P,M',P'$ as in the statement of Proposition \ref{prop_characterizeI-approx}. We can identify $N$ with $P/M$, $N'$ with $P'/M'$ and $\phi$ with the map $\iota$. Then let $Z=P'/M$, $\alpha$ and $\beta$ be the natural maps $P/M\to P'/M$ and $P'/M\to P'/M'$. Clearly $\alpha$ is injective, $\beta$ is surjective and $\phi=\beta \circ \alpha$. Moreover $\Coker \alpha=P'/P$ and $\Ker \beta=M'/M$ are both annihilated by $I$. Hence $\phi$ is an $I$-approximation. Straightforward variations of the arguments given  work in the graded setting as well. 
\end{proof}

Some elementary but useful properties of $I$-approximations  are the following: 

\begin{lem}
\label{lem_approx}
Let $\phi: N\to N'$ be an $I$-approximation. Then $I \Ker \phi=0$ and $I \Coker \phi=0$.  
\end{lem}
\begin{proof}
Using the notation of Definition \ref{defn_approx}, there is an injective map $\Ker \phi\to \Ker \beta$ and a surjective map $\Coker \alpha \to \Coker \phi$.
\end{proof}

\begin{lem}
\label{lem_surjectivemap}
Let $\phi: N\to N'$ be a homomorphism of $R$-modules. 
\begin{enumerate}[\quad \rm(i)]
\item If $\phi$ is injective, then $\phi$ is an $I$-approximation if and only if $I\Coker \phi=0$.
\item If $\phi$ is surjective, then $\phi$ is an $I$-approximation if and only if $I\Ker \phi=0$.
\end{enumerate}
\end{lem}
The proof is immediate from Definition \ref{defn_approx} and Lemma \ref{lem_approx}.
 We introduce now some notation that is useful in dealing with  tensor products.

\begin{rem}
\label{rem_tensor_prod}
Given an ideal $J$ and an $R$-module $M$ then the product $JM$ can be seen as the image of the natural map  $J\otimes_R M\to R\otimes_R M \cong M$. Furthermore  $(R/J)\otimes_R M\cong M/JM$. 
Similarly if $G$ is a free $R$-module and $W\subseteq G$ a submodule we will denote by $WM$ the image of the map 
$$W\otimes_R M\to G\otimes_R M.$$
In particular $GM$ gets identified with $G\otimes_R M$ and 
$$
G/W\otimes_R M\cong GM/WM.
$$
In practice, let  $\{e_i: i \in \Lambda\}$ be a basis for $G$, where $\Lambda$ is an index set, then $GM\cong M^{\Lambda}$ and $WM$ gets identified with  the submodule of $M^{\Lambda}$ generated by elements of the form $(a_im: i\in \Lambda)\in M^\Lambda$ with $m\in M$ and $(a_i: i\in \Lambda)\in W$. Note that, given another free module $G_1$ with a submodule $W_1$ we can identify $W_1(WM)$ with $(W_1W)M$ and with $W(W_1M)$. 
\end{rem}

The following lemma will be employed in the proofs of Theorem  \ref{thm_DS3.5}.

\begin{lem}
\label{lem_Torapprox}
Let $\phi: N\to N'$ be an $I$-approximation. Let $B=G/W$ be an $R$-module where  $G$ is a free $R$-module and $W\subseteq G$ a submodule. 
\begin{enumerate}[\quad \rm(i)]
\item  With the notation of   \ref{rem_tensor_prod},  the natural map $WN \to WN'$ induced by $\phi$ is an $I$-approximation.
\item The map $N\otimes_R B \xrightarrow{\phi\otimes_R B}N'\otimes_R B$ is an $I$-approximation.
\item Assume further that $\phi$ is surjective. Then for all $i\ge 1$, the natural map 
$$
\Tor^R_i(N,B) \xrightarrow{\Tor^R_i(\phi,B)} \Tor^R_i(N',B)
$$ 
is an $I$-approximation.
\end{enumerate}
\end{lem}
\begin{proof}
By definition, we can assume that there exist $R$-modules $U\subseteq Z$ such that $N\subseteq Z, N'=Z/U$, and a commutative diagram
\[
\xymatrix{
&   & Z \ar@{->>}[rd]^{\beta} & \\
&N~ \ar@{^{(}->}[ru]^{\alpha} \ar@{-->}[rr]^{\phi} &  &N' 
}
\]
where $\alpha, \beta$ are respectively natural inclusion and projection maps, such that $\Coker \alpha$ $=Z/N$ and $\Ker \beta=U$ are annihilated by $I$.

(i) Denote by $\phi'$ the natural map $WN\to WN'$. Note that $WN'=(WZ+GU)/GU$, so we have the following commutative diagram
\[
\xymatrix{
&   & WZ \ar@{->>}[rd]^{\beta'} & \\
&WN~ \ar@{^{(}->}[ru]^{\alpha'} \ar@{-->}[rr]^{\phi'} &  &WN' 
}
\]
where $\alpha'$ is injective and $\beta'$ is surjective. Clearly $\Coker \alpha'=WZ/WN$ and $\Ker \beta'=WZ\cap GU$ are annihilated by $I$. Hence $\phi'$ is an $I$-approximation.

(ii) By Remark \ref{rem_tensor_prod}, we may identify $N\otimes_R B$ with $GN/WN$ and $N'\otimes_R B$ with $GN'/WN'$. Denote $\pi=\phi\otimes_R B$. Observe that $GN'/WN'=GZ/(WZ+GU)$, so we have an induced commutative diagram
\[
\xymatrix{
&   & GZ/WN \ar@{->>}[rd]^{\overline{\beta}} & \\
&GN/WN~ \ar@{^{(}->}[ru]^{\overline{\alpha}} \ar@{-->}[rr]^{\pi} &  &GN'/WN' 
}
\]
where $\overline{\alpha}$ is injective and $\overline{\beta}$ is surjective. Clearly $\Coker \overline{\alpha}=GZ/GN$ and $\Ker \overline{\beta}=(WZ+GU)/WN$ are annihilated by $I$. Therefore $\pi$ is an $I$-approximation.

(iii) If $i\ge 1$, replacing $B$ by a suitable syzygy in a free resolution of it, we can assume that $i=1$. Since $\phi$ is surjective, we can write $N=F/U$, $N'=F/V$ where $F$ is a free $R$-module, $U\subseteq V\subseteq F$. Then
\[
\Tor^R_1(N,B) = \Ker(U\otimes_R B\to F\otimes_R B)= (GU\cap WF)/WU.
\]
Similarly 
$$\Tor^R_1(N',B) = (GV\cap WF)/WV.$$
Since $I$ kills $\Ker \phi=V/U$, it holds that $IV\subseteq U\subseteq V$. It remains to apply Proposition \ref{prop_characterizeI-approx}. 
\end{proof}

So combining the results obtained we can give an answer to Question \ref{notzero}: 

\begin{cor} 
\label{kerann}
Let $\phi: N\to N'$ be an $I$-approximation, then $I\Ker(\phi\otimes_R B)=0$ and $I\Coker(\phi\otimes_R B)=0$ for every $R$-module $B$.  
 \end{cor} 

In the remaining of this section $S$ will be  a polynomial ring over $k$ with the graded maximal ideal $\mm$.  
We will define generalized approximations and present regularity bounds obtained by means of them.  

Let $N$ be a finitely generated graded $S$-module and $t\ge 1$ an integer. A    {\it generalized approximation system of degree $t$ for $N$} is a collection of finitely generated graded $S$-modules $N_1,\ldots,N_d$ and degree-preserving homomorphisms $\phi_i: N\to N_i$  such that 
 \begin{enumerate}
 \item $\phi_i$ is an $I_i$-approximation for  $i=1,\ldots,d$,
 \item $\mm^t \subseteq I_1+\cdots+I_d$.
\end{enumerate}

We will say that   $N_1,\ldots,N_d$ is a generalized approximation system of degree $t$ for $N$
 if there are maps $\phi_i$ with the properties above.  

\begin{rem}
Derksen and Sidman \cite[Section 3]{DS2} defined approximation systems as follows. If $N$ is any finitely generated graded $S$-module, then an approximation system of degree $t$ for $N$ is a finite collection of graded $S$-modules $N_1,\ldots,N_d$ and surjections $\phi_i: N\to N_i$ for $i=1,\ldots,d$ which fulfill the following conditions:
\begin{enumerate}
 \item $I_i\Ker \phi_i=0$,
 \item $\mm^t\subseteq I_1+\cdots+I_d$.
\end{enumerate}
From Lemma \ref{lem_surjectivemap} we see that  any approximation system (in the sense of Derksen and Sidman)   is  a generalized approximation system (in our sense).
\end{rem}

Our next goal is to show that the  results in \cite{DS2} on  regularity bounds obtained using approximation systems can be extended to generalized approximation systems. Following \cite{Tr}, a homogeneous element $y\in S$ is called {\it $N$-filter-regular} if the map $N_m \xrightarrow{\cdot y} N_{m+\deg y}$ is injective for all $m\gg 0$. Denote $a_0(N)=\max\{i: H^0_{\mm}(N)_i\neq 0\}$. A useful feature of filter-regular elements is (see \cite[Proposition 1.2]{CH}): If $y$ is a linear form of $S$ which is $N$-filter-regular, then 
$$
\reg N=\max\left\{\reg (N/yN), a_0(N)\right\}.
$$
\begin{thm}
\label{thm_DS3.3}
Let $N$ be a finitely generated graded $S$-module. Let $N_1,\ldots,N_d$ be a generalized approximation system of degree $t$ for $N$ \textup{(}where $d,t\ge 1$\textup{)}. Suppose that $y\in S_1$ is a linear form which is filter-regular with respect to $N$ and $N_i$ for all $i=1,\ldots,d$. 
Then 
$$\reg N\leq \max \left\{  \max_{i=1,\ldots,d} \{ \reg N_i\}+1, \reg (N/yN) \right\}+t-1.$$
\end{thm}

\begin{proof}
The proof of \cite[Theorem 3.3]{DS2} carries over verbatim, taking Lemma \ref{lem_approx} into consideration. Note that the condition $M\to M_i$ is surjective for $i=1,\ldots,d$ in that result is not necessary. The only crucial conditions are that $I_i$ annihilates the kernel of the map $M\to M_i$ for every $i$, and $\mm^t\subseteq I_1+\cdots+I_d$.
\end{proof}
The following result is the most important feature of generalized approximation systems. We will  use it only in the case $t=1$.
\begin{thm}
\label{thm_DS3.5}
Let $n=\dim S$. Let $N$ be a finitely generated graded  $S$-module  and $N_1,\ldots,N_d$ be a generalized approximation system of degree $t$ for $N$. Then: 
$$\reg N\leq   \max\left\{  \max_{i=1,\ldots,d} \{ \reg N_i\}+1,  b \right\}+(t-1)n$$
where $b$ is the largest degree of a minimal generator of $N$. 
\end{thm}
\begin{proof} 
The argument of the proof \cite[Theorem 3.5]{DS2} carries over  verbatim using Theorem \ref{thm_DS3.3} and  that, thanks to Lemma \ref{lem_Torapprox}(ii), the module $N/KN$ admits $N_1/KN_1,\ldots,N_d/KN_d$ as a generalized approximation system of degree $t$ for any homogeneous ideal $K$ of $S$.
\end{proof}


\section{Quadric hypersurfaces}
\label{sect_quadric}
In this section $S$ will denote a standard graded polynomial ring over  a field $k$ with the graded maximal ideal $\mm$. Throughout the section we fix  $I_1,\ldots,I_d$ ideals generated by linear forms, i.e.~$I_i\in \Lscr(S)$, where $d\ge 1$. Furthermore we set  $I=I_1\cdots I_d$.

Denote $[d]=\{1,2,\ldots,d\}$.  For each non-empty index set $A\subseteq [d]$, we let $I_A=\sum_{i\in A}I_i$.  For any $i=1,\dots, d$ we set $J_i=\prod_{j\neq i}I_j$ and more generally for  $1\le i_1<\cdots <i_p\le d$ we set 
$$
J_{i_1,\ldots,i_p} = \prod_{j\not\in \{i_1,\ldots,i_p\} } I_j.
$$
Our main result is the following: 
\begin{thm}
\label{thm_quadric}
 Let $f\in S$ be a quadric  and $R=S/(f)$. Then $R$ is $\Tor$-linear with respect to $\Lscr^\infty(S)$. In particular, $R$ is universally$^*$ Koszul. 
 \end{thm}
 
Indeed, we prove the following more general statement.
\begin{thm}
\label{thm_reg_quotient}
Let $f\neq 0$ be a homogeneous form of positive degree of $S$. Then there is an equality
$$
\creg_S\left(S/I,S/(f)\right) =\reg S/I+\reg S/(f).
$$
\end{thm}

The proof of Theorem \ref{thm_reg_quotient} employs generalized approximation as well as the next lemma.
\begin{lem}
\label{lem_colon_product}
For $i=1,\dots, d$ let $V_i\subseteq  S_1$ be the vector space of linear forms that generates the ideal $I_i$ and $V=\sum_{i=1}^d V_i$. Then the ideal $I:f$ is generated by elements that belong to the polynomial subring $k[V]$ of $S$.
\end{lem} 
To prove Lemma \ref{lem_colon_product}, the following result will  be useful.

\begin{lem}[Conca--Herzog, {\cite[Lemma 3.2 and its proof]{CH}}]
\label{lem_primarydecomp}
One has 
$$I=J_1\cap \dots \cap J_d \cap I_{[d]}^d$$ and,   furthermore, one has  the following \textup{(}possibly redundant\textup{)} primary decomposition:
\[
I= \bigcap_{A\subseteq [d], A\neq \emptyset} I_A^{|A|}. 
\]
\end{lem}

\begin{proof}[Proof of Lemma \ref{lem_colon_product}]
By \ref{lem_primarydecomp} we have that $I:f$ is the intersection of $I_A^{|A|}:f$ where $A\subseteq [d], A\neq \emptyset$.  Since $I_A$ is generated by linear forms it is a prime complete intersection. Hence the powers of $I_A$ are primary and $I_A^{|A|}:f$ is a power of $I_A$.  Therefore $I:f$ is an intersection of powers of $I_A$ as $A$ varies in the set of the non-empty  subsets of $[d]$.  Since each $I_A$ is generated by elements of $k[V]$, the claim follows. 
\end{proof}

Next we come to the proof of Theorem \ref{thm_reg_quotient}.
\begin{proof}[Proof of Theorem \ref{thm_reg_quotient}]
Denote $p=\deg(f)$. Since $\pdim_S S/(f)=1$, we have $\Tor^S_i(S/I,S/(f))=0$ for $i\ge 2$. Together with Theorems \ref{thm_ConcaHerzog} and \ref{thm_Chardin}, the desired equality is equivalent to
\[
\max \left\{\reg \frac{S}{I+(f)}, \reg \Tor^S_1(S/I,S/(f))-1 \right \} \le d+p-2.
\]
By Proposition \ref{prop_compare_reg} for $M=S/I$ and $N=S/(f)$,
\[
\reg \frac{S}{I+(f)} \le \max \left\{d+p-2, \reg \Tor^S_1(S/I,S/(f))-2 \right\}.
\]
Hence it suffices to show that $\reg \Tor^S_1(S/I,S/(f))$ $\le d+p-1$, equivalently $\reg (I:f)/I \le d-1$, noting that $\Tor^S_1(S/I,S/(f))\cong \bigl((I:f)/I\bigr)(-p)$.

For every $i=0,\ldots,d$ and  $0\le t\le d-1$ we set 
\begin{align*}
M_i &= \frac{ (I:f)\cap J_1\cap \cdots \cap J_i}{I},\\
N_t &= \frac{(I:f) \cap J_1 \cap \cdots \cap J_t+ J_{t+1}}{J_{t+1}}.
\end{align*}
Note that, by definition $M_d\subseteq M_{d-1}  \subseteq \dots \subseteq M_{0}$ and $M_i/M_{i+1} \cong N_i$ for all $i=0,\dots,d-1$.
 
We prove by induction on $d\ge 0$  the following statements: 
\begin{enumerate}
\item[(S1)] 
For $i=0,\ldots,d$ one has  $\reg M_i \le d-1$,
\item[(S2)] For every $i=0,\ldots,d-1$ one has  $\reg N_i  \le d-1$. 
\end{enumerate}
In particular,  for $i=0$ in (S1), since $M_0=(I:f)/I$ we obtain the desired inequality.  

The case $d=0$ is immediate. For  $d=1$ both assertions are obviously true:  (S1) holds since for $i=0$, $M_0$ is either $S/I_1$ or $0$ and  $M_1=0$ while (S2) holds since $N_0=0$. 

Now assume that $d\ge 2$ and (S1) and (S2) hold up to $d-1$, we will establish both (S1) and (S2)  for $d$.

First  we prove that  (S2) holds for $d$.   
We treat in full detail the case $1\le i\le d-1$;  a similar argument works for the case $i=0$.

Denote 
$$
P_1=\frac{(J_{i+1}:f)\cap J_{1,i+1} \cap \cdots \cap J_{i,i+1}}{(I:f) \cap J_1 \cap \cdots \cap J_i+J_{i+1}}
$$
and 
$$P_0=\frac{(J_{i+1}:f)\cap J_{1,i+1} \cap \cdots \cap J_{i,i+1}}{J_{i+1}}$$
so that we have an exact sequence
\[
0\to N_i \to P_0 \to P_1\to 0.
\]
Now $P_0$ corresponds to the module $M_i$ associated to  the family of ideals $I_j$ with $j=1,\dots, d$ and $j\neq i+1$. 
Since (S1) holds for $d-1$, we have that
\[
\reg P_0 \le d-2.
\]
Hence, to prove that $\reg N_i\leq d-1$ it suffices to show that $\reg P_1\le d-2$. 

We will define $I_j$-approximation  maps $P_1\to P_{1,j}$ for $j=1,\dots, d$.  For $j=1,\dots, i+1$ one easily checks that $I_j$ annihilates $P_1$ so it suffices to take $P_{1,j}=0$. 
For  $i+2\le j\le d$  set 
\begin{align*}
U_j &= (J_{i+1,j}:f)\cap J_{1,i+1,j} \cap \cdots \cap J_{i,i+1,j},\\
Y_j &= (J_j:f) \cap J_{1,j} \cap \cdots \cap J_{i,j}+J_{i+1,j},
\end{align*} 
so that $Y_j\subseteq U_j$. We check that  

\begin{enumerate}[\quad \rm (a)]
\item $I_jU_j \subseteq (J_{i+1}:f)\cap J_{1,i+1} \cap \cdots \cap J_{i,i+1} \subseteq U_j$,
\item $I_jY_j \subseteq (I:f) \cap J_1 \cap \cdots \cap J_i+J_{i+1} \subseteq Y_j$,
\item setting $P_{1,j}=U_j/Y_j$ one has $\reg P_{1,j}\le d-3$.
\end{enumerate}

Conditions (a)--(b)   are  easy  to check.  To show that  $\reg P_{1,j}\le d-3$ we consider the exact sequence
\[
0\to \frac{Y_j}{J_{i+1,j}} \to \frac{U_j}{J_{i+1,j}} \to P_{1,j} \to 0.
\]

Now we observe that  $U_j/J_{i+1,j}$ is a module of type $M_v$ associated to the family $\{I_u: 1\leq u\leq d, u\neq i+1, j\}$. So by induction hypothesis we have by (S1) that $\reg (U_j/J_{i+1,j})\le d-3$. Similarly, $Y_j/J_{i+1,j}$ is a module of type $N_v$ associated to the family $\{I_u: 1\leq u\leq  d, u\neq j\}$.
Hence by induction hypothesis for (S2)  we have $\reg (Y_j/J_{i+1,j})\le d-2$. Therefore 
$$
\reg P_{1,j}\le \max\left\{\reg \frac{U_j}{J_{i+1,j}}, \reg \frac{Y_j}{J_{i+1,j}}-1 \right\} \le d-3,
$$
as claimed.

By Lemma \ref{lem_colon_product},   the modules $P_1$ and also $P_{1,1}, \ldots,P_{1,d}$ are quotients  of ideals defined in $k[V]$ (notation as in  \ref{lem_colon_product}).  Hence their regularities can be computed over $k[V]$. But the graded maximal ideal of $k[V]$ is $I_1+\dots+I_d$. 
Hence the family of maps $P_1\to P_{1,j}$ defined above is a generalized approximation system of degree $1$. 
 Finally note that $P_1$ is generated in degree at most $d-2$, being a quotient of $P_0$ that has regularity $\leq d-2$. 
Summing up, by Theorem \ref{thm_DS3.5}, we conclude that $\reg P_1\le d-2$. This concludes the proof of (S2) for $d$. 

 Now we proceed to establish (S1) by showing  that  $\reg M_i\le d-1$ by reverse induction on $i$. For $i=d$, by Lemma \ref{lem_primarydecomp}, 
$I:f=(J_1:f)\cap \cdots \cap (J_d:f)\cap (I_{[d]}^d:f)$. Hence if $f\not\in  I_{[d]}$ one has $M_d=0$.  Else if  $f \in I_{[d]}$,  it follows that $J_1\cap \cdots \cap J_d \subseteq I:f$ so that 
\[
M_d=\frac{J_1\cap \cdots \cap J_d}{I}=\frac{J_1\cap \cdots \cap J_d+I_{[d]}^d}{I_{[d]}^d}
\]
All the ideals involved in the expression are generated by elements that belong to $k[V]$ (notation as in \ref{lem_colon_product}), thus the regularity of $M_d$ can be computed over $k[V]$. Since $I_{[d]}$ is the graded maximal ideal of $k[V]$  we get  that $\reg M_d\leq d-1$ by Lemma \ref{lem_finitelength}.  

Finally for $i<d$, since $M_i/M_{i+1} \cong N_i$ one has  
$$\reg M_i \leq \max\{ \reg M_{i+1}, \reg N_i\}$$ 
and $\reg N_i\leq d-1$ because (S2) has been already proved for $d$. Hence we conclude by reverse induction on $i$ that $\reg M_i\leq d-1$.  The proof is completed. 
\end{proof}

We are now ready to prove Theorem \ref{thm_quadric}. 

\begin{proof}[Proof of Theorem \ref{thm_quadric}]
Let $I\in \Lscr^\infty(S)$ be an ideal. Since $\deg f=2$, by Theorem \ref{thm_reg_quotient},
\[
\creg_S(S/I,R)=\reg (S/I)+\reg S/(f)=\reg(S/I)+1.
\]
Hence $R$ is Tor-linear. The last assertion follows from Proposition \ref{prop_sufficient_condition}(ii). This finishes the proof.
\end{proof} 
As the following example shows if $J\subset S$ has a linear resolution over $S$ and $R$ is a quadric hypersurface then, in general, $JR$ need not have a linear resolution over $R$. This show that  Theorem \ref{thm_quadric} cannot be extended to the family of ideals of $S$ with a linear resolution. 

\begin{ex}
Let $S=k[x_1,x_2,y,z]$, $f=x_1^2$ and $R=S/(f)$. Let $J$ be the ideal of 2-minors of the following matrix
\[
\left(
\begin{matrix}
x_1 & x_2  & y \\
x_2 &  0   & z
\end{matrix}
\right)
\]
By computations with CoCoA \cite{CoCoA}, $\reg J=2$ but $\reg_R JR\ge 3$. 
\end{ex}

\subsection*{Acknowledgments}
We are grateful to Luchezar Avramov, Marc Chardin, Srikanth Iyengar and Matteo Varbaro for some useful conversations related to the content of this paper. The first author  was supported by the Istituto Nazionale di Alta Matematica (INdAM). The second author is a Marie Curie fellow of INdAM. The last author is partially supported by the Vietnam National Foundation for Science and Technology Development under grant number 101.04-2016.21.


\begin{thebibliography}{39}
\bibitem{CoCoA}
J. Abbott, A.M. Bigatti and G. Lagorio,
\emph{CoCoA-5: a system for doing Computations in Commutative Algebra}.
Available at  \newblock \verb|http://cocoa.dima.unige.it|.

\bibitem{ACI}
L.L. Avramov, A. Conca, and S.B. Iyengar,
\emph{Subadditivity of syzygies of Koszul algebras}.
Math. Ann. {\bf 361}, no. 1 (2015), 511--534.

\bibitem{AE}
L.L. Avramov and D. Eisenbud,
\emph{Regularity of modules over a Koszul algebra}.
J. Algebra {\bf 153}, 85--90 (1992).

\bibitem{AP}
L.L. Avramov and I. Peeva,
\emph{Finite regularity and Koszul algebras}.
Amer. J. Math. {\bf 123} (2001), 275--281.

\bibitem{Ca} 
G.~Caviglia, 
\emph{Bounds on the Castelnuovo-Mumford regularity of tensor products}. 
Proc.~ Amer.~ Math.~Soc. {\bf 135} (2007), 1949--1957.

\bibitem{Ch1}
M. Chardin,
\emph{On the behaviour of Castelnuovo-Mumford regularity with respect to some functors}.
Preprint (2007), \newblock \verb|http://arxiv.org/abs/0706.2731|.
 

\bibitem{Con1}
A. Conca,
\emph{Universally Koszul algebras}.
Math. Ann. {\bf 317} (2000), 329--346.

\bibitem{Con2}
\bysame,
\emph{Universally Koszul algebras defined by monomials}.
Rend. Sem. Mat. Univ. Padova {\bf 107} (2002), 95--99.

\bibitem{CDR}
A. Conca, E. De Negri and M.E. Rossi,
\emph{Koszul algebras and regularity}.
in {\em Commutative Algebra}, Springer (2013), 285--315.

\bibitem{CH}
A. Conca and J. Herzog,
\emph{Castelnuovo--Mumford regularity of products of ideals}.
Collect. Math. {\bf 54}, 2 (2003), 137--152.

\bibitem{CINR}
A. Conca, S.B. Iyengar, H.D. Nguyen, and T.~R\"omer, 
\emph{ Absolutely Koszul algebras and the Backelin-Roos property}, 
Acta Math. Vietnam. {\bf 40}  (2015) 353-374.    

\bibitem{DS1}
H. Derksen and J. Sidman, 
\emph{A sharp bound on the Castelnuovo--Mumford regularity of subspace arrangements}. 
Adv. Math. {\bf 172} (2) (2002) 151--157.

\bibitem{DS2}
\bysame,
\emph{Castelnuovo-Mumford regularity by approximation}.
Adv. Math. {\bf 188} (2004), 104--123.

\bibitem{EHU}
D. Eisenbud, C. Huneke, and B. Ulrich,
\emph{The regularity of Tor and graded Betti numbers}.
Amer. J. Math. {\bf 128}, no. 3 (2006), 573--605.

\bibitem{GS}
D. Grayson and M. Stillman,
\emph{Macaulay2, a software system for research in algebraic geometry.}
Available at \newblock \verb|http://www.math.uiuc.edu/Macaulay2|.


\bibitem{HI} 
J. Herzog and S.B. Iyengar, 
\emph{Koszul modules}, 
J.~Pure Appl.~Algebra {\bf 201}  (2005), 154-188 
  
 
\bibitem{IR}  
S.B. ~Iyengar and T.~R\"omer,  
\emph{Linearity defects of modules over commutative rings}, 
J.~Algebra  {\bf 322} (2009), 3212-3237.


\bibitem{NgV}
H.D. Nguyen and T. Vu,
\emph{Regularity over homomorphisms and a Frobenius characterization of Koszul algebras}.
J. Algebra {\bf 429} (2015), 103--132.

\bibitem{NgV2}
\bysame,
\emph{Regularity of products over quadratic hypersurfaces}.
Extended Abstracts Spring 2015, Research Perspectives CRM Barcelona vol. {\bf 5}, 129--133, Trends in Mathematics, Springer-Birkh\"auser, Basel, 2016.

\bibitem{R}  
J.-E.~Roos,
\emph{Good and bad Koszul algebras and their Hochschild homology},  
J.~Pure Appl.~Algebra  {\bf  201} (2005), 295-327.


\bibitem{S} 
J.~Sidman, 
\emph{On the Castelnuovo-Mumford regularity of products of ideal sheaves}. 
Adv.~Geom. {\bf 2} (2002), no. 3, 219--229.

\bibitem{Tr}
N.V. Trung,
\emph{Reduction exponent and degree bounds for the defining equations of a graded ring}.
Proc. Amer. Math. Soc. {\bf 102} (1987), 229--236.

\end{thebibliography}
\end{document}